\documentclass[a4paper]{article}
\usepackage{lmodern}
\usepackage{stmaryrd,amsmath,amssymb,ifthen}
\usepackage{amsthm}
\usepackage[mathjax]{lwarp}

\newtheorem{prop}{Proposition}
\newtheorem{thm}[prop]{Theorem}
\newtheorem{cor}[prop]{Corollary}
\newtheorem{lem}[prop]{Lemma}

\theoremstyle{definition}

\newtheorem{defn}[prop]{Definition}
\newtheorem{example}[prop]{Example}

\newtheorem{conj}[prop]{Conjecture}

\theoremstyle{remark}

\newcommand{\mathmacro}[1]{#1\CustomizeMathJax{#1}}

\mathmacro{
\newcommand{\N}{\mathbb{N}}
\newcommand{\Q}{\mathbb{Q}}
\newcommand{\R}{\mathbb{R}}
\newcommand{\Z}{\mathbb{Z}}
\newcommand{\C}{\mathbb{C}}
\newcommand{\F}{\mathbb{F}}
\newcommand{\Gal}{\mathrm{Gal}}
\newcommand{\mbG}{\mathbf{G}}

\newcommand{\GO}{\mathrm{GO}}
\newcommand{\tA}{\mathrm{A}}
\newcommand{\tB}{\mathrm{B}}
\newcommand{\tC}{\mathrm{C}}
\newcommand{\tD}{\mathrm{D}}

\newcommand{\mP}{\mathbb{P}}
\newcommand{\mPb}{\bar{\mathbb{P}}}
\newcommand{\lcm}{\mathrm{lcm}}
\newcommand{\GL}{\operatorname{GL}}
\newcommand{\PGL}{\operatorname{PGL}}
\newcommand{\SL}{\operatorname{SL}}
\newcommand{\PSL}{\operatorname{PSL}}
\newcommand{\Sp}{\operatorname{Sp}}

\newcommand{\GU}{\operatorname{GU}}

\newcommand{\SU}{\operatorname{SU}}
\newcommand{\PSU}{\operatorname{PSU}}
\newcommand{\Orth}{\operatorname{O}}

\newcommand{\e}{\mathrm e}
\newcommand{\I}{\mathrm i}
\newcommand{\ep}{\varepsilon}
}

\title{An Ennola duality for subgroups of groups of Lie type}

\numberwithin{prop}{section}
\numberwithin{equation}{section}

\author{David A.\ Craven, University of Birmingham}
\date{28th October, 2021}

\begin{document}
\maketitle

\setcounter{tocdepth}{3}
\begin{abstract} We develop a theory of Ennola duality for subgroups of finite groups of Lie type, relating subgroups of twisted and untwisted groups of the same type. Roughly speaking, one finds that subgroups $H$ of $\mathrm{GU}_d(q)$ correspond to subgroups of $\mathrm{GL}_d(-q)$, where $-q$ is interpreted modulo $|H|$. Analogous results for types other than $\mathrm A$ are established, including for those exceptional types where the maximal subgroups are known, although the result for type $\mathrm D$ is still conjectural. Let $M$ denote the Gram matrix of a non-zero orthogonal form for a real, irreducible representation of a finite group, and consider $\alpha=\sqrt{\det(M)}$. If the representation has twice odd dimension, we conjecture that $\alpha$ lies in some cyclotomic field. This does not hold for representations of dimension a multiple of $4$, with a specific example of the Janko group $\mathrm J_1$ in dimension $56$ given. (This tallies with Ennola duality for representations, where type $\mathrm D_{2n}$ has no Ennola duality with ${}^2\mathrm D_{2n}$.)
\end{abstract}
\section{Introduction}

In \cite{ennola1963}, Ennola described a conjectural way to extract the character table of $\GU_n(q)$ from that of $\GL_n(q)$, broadly by replacing `$q$' by `$-q$' and making some other alterations. (This was proved in \cite{kawanaka1985}.) Ennola duality later became the principle that representation-theoretic information about twisted groups of Lie type can be inferred from the untwisted groups via the same method (although for type $\tD_n$ for $n$ even the `Ennola dual' is not ${}^2\tD_n$ but $\tD_n$ itself, so information cannot be passed between them). For example, the set of unipotent degrees of a group of Lie type satisfy this duality, with Ennola duality inducing a self-bijection if there is no corresponding twisted group.

The purpose of this note is to record a phenomenon that appears not to have been noticed before in the literature, that Ennola duality extends to subgroups. What we mean by this is that one can extract from the list of maximal subgroups of $\SL_n(q)$ (most of) the maximal subgroups of $\SU_n(q)$, again by replacing `$q$' by `$-q$'.

The precise statement of Ennola duality for type $\tA$ is given in Theorem \ref{thm:ennolatypeA} below, but this gives a flavour for the statement.

\begin{thm} Let $H$ be a finite group and let $\chi$ be an irreducible character for $H$ of degree $d$. Suppose that $q$ and $q'$ are prime powers such that $q'\equiv -q\bmod |H|$, and $q$ and $|H|$ are coprime. If $H$ embeds in $\GL_d(q)$ with (Brauer) character $\chi$ then $H$ embeds in $\GU_d(q')$ with (Brauer) character $\chi$.
\end{thm}

Theorem \ref{thm:ennolatypeA} is precise about the integers that can be used instead of $|H|$ for the congruences, and can be used with tables such as those in \cite{bhrd} to immediately read off subgroups of $\GU_d(q)$. (If $\chi$ is real-valued and $H$ embeds in $\GL_d(q)$, then it embeds in $\GL_d(q')$, $\GU_d(q)$ and $\GU_d(q')$ with the same character, so the content is for $\chi$ not real-valued.)

A similar statement, Proposition \ref{prop:ennolatypeB}, holds for types $\tB$ and $\tC$, but here $H$ always embeds in $G(q')$ whenever it embeds in $G(q)$, with the same hypotheses. For now, type $\tD$ remains somewhat more difficult to work with, and a form of Ennola duality holds, at least for integral representations (Theorem \ref{thm:EnnolatypeD}). For type $\tD_{28}$ we give an example to show that there is no number $m$ for which the embedding of the Janko group $\mathrm J_1$ into $\GO_{56}^+(q)$ depends only on the congruence of $q$ modulo $m$. However, for $\tD_n$ with $n$ odd, we have found no similar examples, and if Conjecture \ref{conj:EnnolatypeD} is true there are no such example. For $\tD_n$ and $n$ odd, if Conjecture \ref{conj:EnnolatypeD} holds (and it does for all integral representations) we simply replace $q$ by $-q$ modulo some integer to switch between $\GO^+$ and $\GO^-$, as with other types.

Notice that this example of $\mathrm J_1$ in dimension $56$ also throws up another point. It is known that the $p$-modular reduction can interact badly with algebraic conjugacy when $p$ divides the order of a finite group. This example shows similar behaviour even when $p$ is prime to the order, and in particular means that the type of symmetric bilinear form over $\F_q$ cannot, in a strict sense, be deduced from the character table of a finite group.

There is also no complete proof for exceptional groups. All known subgroups of exceptional groups satisfy the same Ennola-like behaviour, but the maximal subgroups of $E_8(q)$ are not yet known, and for $E_7(q)$ the conjectured list in \cite{craven2021un} (see Table \ref{tab:e7} below) is not known to be complete. We describe the version of Ennola duality that exists for these groups in Section \ref{sec:exceptional}. Again, this boils down to replacing $q$ in the tables in Section \ref{sec:exceptional} with $-q$ to obtain the table of the twisted version if there is one, or the same table again if there is not.

\medskip

Although it does not appear that the ideas in this work have appeared in this generality before, the work on classifying low-dimensional subgroups of classical groups, particularly orthogonal groups, in \cite{bhrd,rogersphd,schroederphd}, is distinctly reminiscent of it (if independent, as the author only learned of their ideas after writing this). For example \cite[Lemma 4.9.39]{bhrd}, which was mentioned (possibly for the first time in the literature) in \cite{turull1993}, is given a more general treatment in Proposition \ref{prop:typeDinteger}. A posteriori, one might see this article as offering a general framework for the results in those papers.

\medskip

In the next section we collect some preliminary definitions and results, and discuss fields of values for characters and embedding into linear groups over finite fields. The section after considers classical groups, and Section \ref{sec:exceptional} deals with exceptional groups.

We have written this paper to minimize the use of class field theory, although for groups of type $\tD$ it seems difficult to prove anything meaningful without at least basic results from it. This means a few preliminary lemmas are included that are standard in class field theory, but we feel that making the paper as self-contained as possible is worth it.

\medskip

The author would like to thank Gunter Malle for reading through a preliminary version of the manuscript, and particularly Elie Studnia for providing a proof of Proposition \ref{prop:allfinitefields}(i). He would also like to thank the referee for a number of helpful comments.

This work was partially supported by a Royal Society University Research Fellowship.

\section{Preliminaries, Brauer characters and fields of values}

In this article, $H$ denotes a finite group and $\chi$ is an irreducible ordinary character of $H$. Write $d=\chi(1)$. Without loss of generality in this article, we may assume that $\chi$ is faithful. Let $p$ be a prime such that the reduction modulo $p$ of $\chi$ remains irreducible. (This includes all primes not dividing $|H|$ but is in many cases strictly larger.) Write $\mP$ for the set of all such primes $p$, and $\mPb$ for the set of all powers of members of $\mP$. For a prime $p$, let $|H|_{p'}$ denote the $p'$-part of $|H|$.

If $p$ is any prime, we let $k$ denote an algebraically closed field of characteristic $p$. Since one of the main issues in this article is that `reduction modulo $p$' does not necessarily behave well with respect to the sign of bilinear forms, we have to be strict about what we mean. We fix a primitive root of unity of order $n=|H|_{p'}$ in $k$, and identify $\chi$ with an irreducible Brauer character modulo $p$ via the assignment of $\e^{2\pi\I/n}$ to that root of unity. This allows us to talk about subgroups of $\GL_d(k)$ isomorphic to $H$ with Brauer character $\chi$.

For $q$ a power of $p$, let $F_q$ denote the field automorphism of $k$ given by $x\mapsto x^q$. Given $p$, let $\bar H$ denote a copy of $H$ in $\GL_d(k)$ with Brauer character $\chi$, and note that $\bar H$ is necessarily an absolutely irreducible subgroup of $\GL_d(q)$ whenever $\bar H$ is a subgroup of it (since $\chi$ remains irreducible as a Brauer character).

The meanings of the terms $\Sp_d(q)$, $\GO_d^+(q)$, and so on, are the full symmetry group of the appropriate form in $\GL_d(q)$. The \emph{exponent} of a finite group is the lowest common multiple of all orders of all elements of it. Let $\zeta_n=\e^{2\pi\I/n}$ be a primitive $n$th root of unity in $\C$.

\medskip

The first lemma is well known, and gives the minimal field of definition for a finite group embedding in characteristic $p\neq 0$. (The case for characteristic $0$ has no easy answer.) See, for example, \cite[Section 5]{abc}.

\begin{lem}\label{lem:fieldvalues} Let $p\in \mP$ be a prime. The subgroup $\bar H$ of $\GL_d(k)$ is conjugate to a subgroup of $\GL_d(q)$ if and only if the traces (evaluated in $k$) of all elements of $\bar H$ lie in $\F_q$. Thus $H$ embeds in $\GL_d(q)$ with character $\chi$ if and only if, under a map sending a primitive $|H|$th root of unity over $\C$ to one over $k$, the Brauer character values of $\chi$ lie in $\F_q$.

In particular, $\bar H$ is conjugate to a subgroup of $\GL_d(q)$ for $q$ such that the exponent of $H$ divides $q-1$.
\end{lem}

The second statement follows from the first since the trace of a matrix is the sum of its eigenvalues, and the eigenvalues are roots of unity of order dividing the exponent of $H$.

There is a subtlety with this lemma, that is not often explicitly mentioned. Let $p=5$ and suppose that $A$ is a $7$-dimensional matrix with eigenvalues $\omega$ six times, and $\omega^2$ once, where $\omega^3=1$. The trace, $-1$, certainly lies in $\F_5$, but there is no matrix in $\GL_7(5)$ with those eigenvalues. So the claim is certainly only true for \emph{irreducible} Brauer characters, and tells us something about the eigenvalues of semisimple elements in irreducible subgroups of $\GL_d(k)$, where $k$ has characteristic $p$.

\begin{lem}\label{lem:matingln} Let $A$ be a matrix in $\GL_d(k)$. The multiset of eigenvalues of $A$ is invariant under the Frobenius endomorphism $F_q$ if and only if $A$ is conjugate to an element of $\GL_d(q)$.
\end{lem}

For a proof, a semisimple conjugacy class of $\GL_d(k)$ is determined the multiset of eigenvalues of a representative matrix. The Frobenius endomorphism permutes the classes, and this action is determined by its action on the eigenvalues. Thus a conjugacy class is stabilized by the Frobenius endomorphism if and only if the multiset of eigenvalues is. Now apply, for example, \cite[Theorem 26.7]{malletesterman} to obtain a semisimple element of $\GL_d(k)$ fixed by the Frobenius endomorphism whenever the class is stabilized by it.

\medskip

So Lemma \ref{lem:fieldvalues} tells us that, for a simple $kH$-module $M$, $M$ is realizable over $\F_q$ if and only if the multiset of eigenvalues of each element of $H$ on $M$ is invariant under $F_q$. The same will therefore be true of their lifts to $\C$ when performing the Brauer character construction, that they are invariant under $\zeta_n\mapsto\zeta_n^q$, where $\zeta_n$ is an appropriate root of unity and the map is on the field $\Q(\zeta_n)$.

\medskip

We now give a fundamental definition for this article.

\begin{defn} For a given $H$ and $\chi$, a positive integer $n$ is a \emph{defining modulus} if, whenever $q$ is a prime power congruent to $1$ modulo $n$ (and in $\mPb$), the corresponding subgroup $\bar H$ is conjugate to a subgroup of $\GL_d(q)$.
\end{defn}

This is closely related to the concept of a defining modulus in class field theory: the defining modulus of a character is the defining modulus of the smallest subfield of $\mathbb C$ containing its character values.

It is not obvious from this definition, but defining moduli appear in all tables of maximal subgroups of classical groups. The next lemma teases out the first relevance for the problem.

\begin{lem}\label{lem:defmodcyclotomic} An integer $n$ is a defining modulus for $H$ and $\chi$ if and only if the values of $\chi$ lie in the $n$th cyclotomic field $\Q(\zeta_n)$
\end{lem}
\begin{proof} Let $n'=\lcm(|H|,n)$. Abbreviate $\zeta_n$ by $\zeta$ and $\zeta_{n'}$ by $\xi$.

Suppose first that the character values of $\chi$ lie in $\Q(\zeta)$. Let $q$ (in $\mPb$) be a prime power congruent to $1$ modulo $n$. In order to show that $\bar H$ is conjugate to a subgroup of $\GL_d(q)$, it suffices (by the discussion after Lemma \ref{lem:matingln}) to show that the Brauer character values of $\chi$ are fixed under map $\xi\mapsto\xi^q=\xi$. This is obviously true, and so the result follows.

\medskip

For the converse, suppose that $n$ is a defining modulus. Since $\chi(x)$ lies in $\Q(\xi)$ for all $x\in H$, it suffices to show that $\chi(x)$ is centralized by all elements of $\mathrm{Gal}(\Q(\xi)/\Q(\zeta))$. Viewed as maps in $\mathrm{Gal}(\Q(\xi)/\Q)$, these are given by $\xi\mapsto \xi^a$ for $a\equiv 1\bmod n$. Let $q$ (in $\mPb$) be a prime power congruent to $a$ modulo $n'$, and note that $q\equiv 1\bmod n$. Thus by assumption $\bar H$ is conjugate to a subgroup of $\GL_d(q)$, so the lifts of the eigenvalues of elements of $\bar H$ to $\Q(\xi)$ are invariant under the map $\xi\mapsto \xi^q=\xi^a$. But this is what is required, and the other direction holds.
\end{proof}

This alternative interpretation of defining moduli has the following easy consequence.

\begin{cor} All defining moduli for a fixed $H$ and $\chi$ are multiples of the smallest defining modulus, called the \emph{conductor}. Furthermore, a positive integer $m$ is a defining modulus if and only if $\gcd(m,|H|)$ is.
\end{cor}

This is proved simply by invoking Lemma \ref{lem:defmodcyclotomic} and taking the intersections of the appropriate cyclotomic fields. Like defining moduli, the name `conductor' comes from class field theory.

\medskip

We now introduce defining residues, which explains the choice of this definition. Given $H$ and $\chi$, and a defining modulus $n$, let $I_n$ denote the set of all $i$ between $1$ and $n-1$, and prime to $n$, such that the field isomorphism on $\Q(\zeta_n)$ induced by $\zeta_n\mapsto \zeta_n^i$  fixes each value of $\chi$. The set $I_n$ is called the set of \emph{defining residues}. Notice that the set of defining residues is a subgroup of the unit group $(\Z/n\Z)^\times$. We often abuse terminology and say that an integer $a$ lies in $I_n$ if $a\bmod n$ lies in $I_n$. Particularly, we do this with negative numbers, so $-i$ lies in $I_n$ for $1\leq i\leq n-1$ to mean $n-i\in I_n$.

\begin{thm} For a fixed $H$ and $\chi$, let $n$ be a defining modulus dividing $|H|$. For all prime powers $q$ (in $\mPb$), $\bar H$ is conjugate to a subgroup of $\GL_d(q)$ if and only if the congruence class of $q$ modulo $n$ appears in $I_n$. In particular, if $q'$ (in $\mPb$) is another prime power and $q\equiv q'\bmod n$, then one may embed $H$ in $\GL_d(q)$ with character $\chi$ if and only if one can do the same with $\GL_d(q')$.
\end{thm}
\begin{proof} This proof has the same ideas as that of Lemma \ref{lem:defmodcyclotomic}, so we are a little briefer this time. The subgroup $\bar H$ is conjugate to a subgroup of $\GL_d(q)$ if and only if the eigenvalues of elements of $\bar H$, lifted to $\Q(\zeta_n)$, are invariant under the map $\zeta_n\mapsto\zeta_n^q$. This statement is equivalent to the statement that $q\bmod n$ lies in $I_n$ by definition of $I_n$, and thus $q\bmod n$ lies in $I_n$ if and only if $\bar H$ is conjugate to a subgroup of $\GL_d(q)$, as claimed.
\end{proof}

Finally, we need a well-known result on when $\bar H$, which is already conjugated to lie in $\GL_d(q)$, also lies in another classical group. Basically, the answer is `whenever they stabilize the appropriate form', so there is no need to worry that we might need to increase $q$ to find $\bar H$ inside the classical group.

\begin{lem}\label{lem:conjtosubgroup} Suppose that $\chi$ is real-valued, irreducible, and $d>1$.
\begin{enumerate}
\item If $d$ is odd then $\bar H$ is conjugate to a subgroup of $\GO_d(q)$ if and only if it is conjugate to a subgroup of $\GL_d(q)$.

\item If $d$ is even and $\chi$ has indicator $-1$, then $\bar H$ is conjugate to a subgroup of $\Sp_d(q)$ if and only if it is conjugate to a subgroup of $\GL_d(q)$.

\item If $d$ is even and $\chi$ has indicator $+1$, $\bar H$ is conjugate to a subgroup of exactly one of $\GO_d^+(q)$ and $\GO_d^-(q)$ via $\chi$ if and only if it is conjugate to a subgroup of $\GL_d(q)$.
\end{enumerate}
\end{lem}
\begin{proof} Since the exterior or symmetric square of a simple module has a unique trivial submodule over any field, any module with character $H$ stabilizes a unique bilinear form (quadratic form for $p=2$). This is enough to prove existence of $H$ in symplectic and orthogonal groups.

In the third case, the uniqueness of the form means that $H$ cannot lie in both plus and minus type orthogonal groups.
\end{proof}

There is a nice way to see whether a character is real-valued from the set $I_n$.

\begin{lem}\label{lem:charselfudal} Let $n$ be a defining modulus for a given $H$ and $\chi$, and $I_n$ the defining residues. The following are equivalent:
\begin{enumerate}
\item $\chi$ is real-valued;
\item $-1$ lies in $I_n$;
\item $I_n$ is closed under taking negatives, i.e., $-i\in I_n$ whenever $i\in I_n$.
\end{enumerate}
\end{lem}
\begin{proof} The equivalence of the second and third statements follows immediately from the fact that $I_n$ is a group under multiplication. Also, $-1$ lies in $I_n$ if and only if the map $\zeta_n\mapsto \zeta_n^{-1}$ (i.e., complex conjugation) stabilizes $\chi$. But this is true if and only if $\chi$ is real-valued.
\end{proof}

%
%
%

\section{Ennola duality for classical groups}

Our first goal is to state and prove Ennola duality for subgroups of $\GL$ and $\GU$. It is not quite as simple as replacing $q$ by $-q$, but it is close. If $\chi$ is real-valued then $H$ embeds in a symplectic or orthogonal group, hence embeds in both $\GL_d(q)$ and $\GU_d(q)$ whenever it embeds in one of them. Thus we are most interested in the case where $\chi$ is not real-valued, or equivalently that there exists $i\in I_n$ such that $-i\not\in I_n$ by Lemma \ref{lem:charselfudal}.

\begin{thm}\label{thm:ennolatypeA} Let $H$ be a finite group and let $\chi$ be a faithful irreducible character of $H$. Let $n$ be a defining modulus and $I_n$ the set of defining residues. For a power $q$ of a prime $p$ in $\mP$, the group $H$ embeds in $\GU_d(q)$ via $\chi$ if and only if $-q\bmod n$ lies in $I_n$.
\end{thm}
\begin{proof} By the remark preceding the theorem, we may assume that $\chi$ is not real-valued, i.e., $-1\not\in I_n$. By assumption, for any $x\in \bar H$, the eigenvalues of $x$ are permuted by the map $\zeta_n\mapsto\zeta_n^i$, but not by the map $\zeta_n\mapsto \zeta_n^{-i}$, the composition of the former map and $\zeta_n\mapsto\zeta_n^{-1}$. Hence $\zeta_n\mapsto \zeta_n^{-i}$ maps $\chi$ to $\bar\chi$, the complex conjugate, of $\chi$, and $\bar\chi\neq \chi$. Since the graph automorphism acts as inverse transpose, for $q\equiv -i\bmod n$, the composition $\sigma$ of $F_q$ with the graph automorphism stabilizes $\chi$.

This is enough to prove that $\bar H$ is conjugate to a subgroup of $\GU_d(q)$. To see this, note first that, since $\bar H$ is unique up to conjugacy subject to having character $\chi$, $\sigma$ stabilizes the $\GL_d(k)$-conjugacy class containing $\bar H$, so normalizes some member of it by \cite[Theorem 21.11]{malletesterman}. Then we apply \cite[Lemma 1.8.6]{bhrd}, which states that $\sigma$ must therefore act as some element of $\GL_d(k)$ that normalizes (a conjugate of) $\bar H$. Then we apply \cite[Corollary 21.8]{malletesterman}, which implies that $\sigma$ then centralizes some conjugate of $\bar H$, and thus some conjugate of $\bar H$ lies inside the fixed points of $\sigma$, namely $\GU_d(q)$.

To see the converse, if both $F_q$ and the product with the graph automorphism centralize (a conjugate of) $H$ then the graph automorphism acts as an element of the normalizer in $\GL_d(k)$ of $H$. Such an element stabilizes $\chi$, but the graph automorphism maps $\chi$ to its complex conjugate. Thus $\chi$ is real-valued and so $-1\in I_n$, a contradiction.
\end{proof}

We now move on to the other classical groups. Types $\tB$ and $\tC$ are easy, since there is no twisted type to be concerned about.

\begin{prop}\label{prop:ennolatypeB} Given $H$ and $\chi$, suppose that $\chi$ is real-valued. Furthermore, suppose that $\chi$ has indicator $-1$, or $\chi(1)$ is odd. If $n$ is a defining modulus and $I_n$ the defining residues, then $i\in I_n$ if and only if $-i\in I_n$.

Consequently, if $q\equiv -q'\bmod n$, and $q$ and $q'$ are prime powers in $\mPb$, $H$ embeds in $\Sp_d(q)$ if and only if $H$ embeds in $\Sp_d(q')$ (via $\chi$), and similarly for $\GO_d(q)$ and $\GO_d(q')$.
\end{prop}

The proof is immediate, from Lemma \ref{lem:charselfudal}.

\subsection{Groups of type $\tD$}

Suppose that $\chi$ has Frobenius--Schur indicator $+1$, so that $\chi$ is the character of a real representation. Let $n$ be a defining modulus and $q\bmod n\in I_n$. Then $\bar H$ is conjugate to a subgroup of $\GO^\ep_d(q)$ for some $\ep=\pm 1$ by Lemma \ref{lem:conjtosubgroup}, but deciding which is a significant issue. It comes down to knowing the determinant of the Gram matrix for the bilinear form stabilized by $\bar H$. For groups of type $\tD$ we restrict to the case where the prime powers in $\mPb$ are odd, because this description in terms of Gram matrices does not apply to $p=2$.

\begin{defn} Given $H$, and $\chi$ with indicator $+1$, a defining modulus $m$ is a \emph{discriminating modulus} if the set $I_m$ can be partitioned into two disjoint subsets $I_m^+\cup I_m^-$, such that $\bar H$ is conjugate to a subgroup of $\GO^+_d(q)$ ($q$ odd in $\mPb$) if and only if $q\bmod m$ lies in $I_m^+$. The sets $I_m^+$ and $I_m^-$ are the positive and negative \emph{discriminating residues}.
\end{defn}

If $\chi$ is afforded by a real representation then not all defining moduli are discriminating moduli. This can be seen most obviously with representations over $\Z$, where $1$ is a defining modulus (assuming $q$ is odd), but certainly there are subgroups of $\GO_d^-(q)$ that come from $\Z$-representations, such as the alternating group $A_7$, which lies in $\GO_6^-(q)$ for $q\equiv 3,5,6\bmod 7$. As with defining moduli, the set of discriminating moduli is of the form $m\N$ for some minimal modulus. However, it is not clear that discriminating moduli always exist. To examine this problem we need a bit more background.

First, changing basis for the bilinear form multiplies the determinant by a square, so we can only determine the determinant up to a square. For example, if the determinant lies in $\Z$, then it can be given by a square-free integer $m$. A representative of the determinant in $(F^\times)/(F^\times)^2$ (where the matrix is defined over $F$) is called the \emph{discriminant} of the form. For fields $\F_q$, the set $|(\F_q^\times)/(\F_q^\times)^2|$ has order $2$.

Over $\Q$, the minimal square-free integer $m$ must be odd: the reduction modulo $2$ of the matrix $M$ of the form must be the matrix of a symmetric, hence skew-symmetric bilinear form, and so modulo the radical it has even dimension. Thus an even number of eigenvalues must be even, and $4\nmid m$. Also, $m$ is positive:  by extending the field to $\R$, we may change basis so that the matrix of the form is diagonal with entries $\pm 1$. The subspace spanned by basis elements with norm $1$ yields an invariant subspace, as does that with norm $-1$. Thus one of these subspaces is zero as the subgroup $\bar H$ of $\GL_d(k)$ is irreducible. But the dimension $d$ is even, so the determinant is always positive. The statement that all eigenvalues of $M$ must be positive (or all must be negative) holds for any $H$ and $\chi$, not just $\Z$-representations.

One can also see that if $m$ is divisible by a prime $r$ then either all eigenvalues of the matrix are divisible by $r$, in which case we can multiply by a scalar to remove them, or the reduction modulo $r$ of the form has a non-zero radical. Thus, in particular, the reduction modulo $r$ of $\chi$ cannot be an irreducible Brauer character. This proves that $r$ a product of primes not belonging to $\mP$.

\medskip

We now need the discriminant of the form for the two orthogonal groups $\GO_d^\ep(q)$. If $d$ is divisible by $4$, or $d\equiv 2\bmod 4$ and $q\equiv 1\bmod 4$ then the discriminant for $\GO_d^+(q)$ is a square, and if $d\equiv 2\bmod 4$ and $q\equiv 3\bmod 4$ then the determinant is a non-square. The discriminant for $\GO_d^+(q)$ is a square in $\F_q$ if and only if the discriminant for $\GO_d^-(q)$ is a non-square. Thus if $\bar H$ is a subgroup of $\GL_d(q)$ and we can determine the discriminant for the invariant bilinear form then we know which of the two orthogonal groups $\bar H$ embeds into.

Ennola duality for type $\tD_n$ with $n$ odd goes through as before, at least if the discriminant over $\Q$ is an odd positive integer. (This therefore includes the case where the representation is over $\Z$.) Although this proposition is later subsumed into a more general result, representations over $\Z$ are common enough for it to be useful to have the formulae derived in the proof.

\begin{prop}\label{prop:typeDinteger} Let $H$ and $\chi$ be given, and suppose that $\chi$ has indicator $+1$, arising from a representation over $\Z$. Suppose that $d\equiv 2\bmod 4$, and that the discriminant of the form is an odd positive integer $m$. Let $n$ be a multiple of $4m$ such that $n$ is a defining modulus for $\chi$.
\begin{enumerate}
\item The number $n$ is a discriminating modulus for $H$ and $\chi$.
\item The set $I_n^+$ is a subgroup of index $2$ in $I_n$ and $I_n^-=\{-i:i \in I_n^+\}$.
\end{enumerate}
\end{prop}
\begin{proof} Let $q\in \mPb$ be a power of a prime $p$. In order to determine the discriminant of the form for $\bar H$, which is the reduction modulo $p$ of $m$, it suffices to check whether $m$ has a square root in $\F_q$. If $q$ is an even power of $p$, this is always the case. If $q$ is an odd power of $p$, then this is equivalent to whether $m$ has a square root in $\F_p$, which is given by the Legendre symbol $(m/p)$. Note we will also need to take into account the discriminant for the standard form for $\GO_d^\ep(q)$, so we need the congruence modulo $4$ as well, which is the Legendre symbol $(-1/p)$. Thus $\bar H$ is conjugate to a subgroup of $\GO_d^+(q)$ if and only if $q\bmod m$ lies in $I_m$, and either $q$ is a square or $(-m/p)=1$.

We consider the second condition. Writing $m=m_1\ldots m_r$ with each $m_i$ prime, by quadratic reciprocity,
\begin{equation} \left(\frac{-m}{p}\right)=(-1)^{(p-1)/2}\prod_{i\in I} \left(\frac{m_i}{p}\right)=(-1)^{(p-1)/2}(-1)^{(m-1)(p-1)/4}\prod_{i\in I} \left(\frac{p}{m_i}\right)=(-1)^{(m+1)(p-1)/4}\prod_{i\in I} \left(\frac{p}{m_i}\right).\end{equation}
If $m\equiv 3\bmod 4$ then this reduces to $\prod (p/m_i)$, and whether this is $1$ depends only on the congruence of $p$ modulo $\prod m_i=m$. Furthermore, this set of congruences is a subgroup of index $2$ in $\Z/m\Z^\times$, and contains all squares in $\Z/m\Z^\times$ (so $q\bmod m$ lies in it whenever $q$ is a square). Also, $p\bmod m$ lies in this set if and only if $p^i$ lies in it for any odd $i$. Since $n$ is a multiple of $m$, the same holds modulo $n$ as well. Thus $I_n^+$ is well defined, and has index $2$ in $I_n$.

Finally, if $p\equiv -1\bmod n$ then each $(1/m_i)$ is $1$ if $m_i\equiv 1\bmod 4$, and $-1$ if $m_i\equiv 3\bmod 4$. Thus $-1$ does not lie in $I_n^+$ since $m\equiv 3\bmod 4$, and the result holds.

\medskip

Thus we assume that $m\equiv 1\bmod 4$, so the formula above reduces to $(-1)^{(p-1)/2}\prod_i (p/m_i)$. Now whether this is $1$ depends only on the congruence of $p$ modulo $4m$. Similarly, this set is a subgroup of index $2$ in $\Z/4m\Z^\times$, contains all squares, and $p\bmod 4m$ lies in the set if and only if $p^i\bmod m$ lies in it for any odd $i$. Again, if $p\equiv -1\bmod 4m$ then $p\equiv 3\bmod 4$, so the sign at the front is $-1$, and the product evaluates to $-1$. Thus $-1\not\in I_n^+$. This completes the proof.
\end{proof}

If $4\mid d$ then the same proof works, but we do not obtain that $-1$ lies in $I_n^-$, and so $I_n^+$ has index at most $2$ in $I_n$. (An example where $I_n^+=I_n$ is $\Omega_8^+(2)$, which lies in $\Omega_8^+(p)$ for all primes, and not in $\Omega_8^-(p)$.) The formula for whether $i\in I_n$ lies in $I_n^+$ now becomes
\begin{equation} \left(\frac{m}{p}\right)=\prod_{i\in I} \left(\frac{m_i}{p}\right)=(-1)^{(m-1)(p-1)/4}\prod_{i\in I} \left(\frac{p}{m_i}\right).\end{equation}

\medskip

We see in the proof of Proposition \ref{prop:typeDinteger} that the important thing was that the square root of the discriminant lies in $\F_q$ if and only if $q$ lies in some set modulo some integer. This motivates the following definition.

\begin{defn}\label{defn:squareroot} Given $H$ and $\rho$ a representation affording $\chi$, let $\alpha$ denote the determinant of a Gram matrix of the form stabilized by $\rho$. We say that $\rho$ is \emph{root cyclotomic} if the square root of $\alpha$ lies in some cyclotomic field, i.e., $\Q(\sqrt\alpha)/\Q$ is an abelian extension.
\end{defn}

If $\rho$ is root cyclotomic then we can prove a version of Ennola duality for $H$ even when $\rho$ is not over the integers. The proof uses class field theory, which appears not to be avoidable at this point.

First, we need to extend results from class field theory about splitting of polynomials from fields $\F_p$ to fields $\F_q$, which of course is of interest to us here. This does not appear to be a standard part of class field theory, so we have to do it ourselves here.

\begin{prop}\label{prop:allfinitefields} Let $\alpha$ be an element of a cyclotomic field. Let $f$ be its minimal polynomial over $\Z$, and $m$ be minimal such that $\alpha$ lies in the $m$th cyclotomic field. Suppose that the leading coefficient of $f$ is only divisible by primes dividing $m$. Suppose that $p$ and $p'$ are primes not dividing $m$, and let $q$ and $q'$ be powers of $p$ and $p'$ respectively.
\begin{enumerate}
\item If $q\equiv q'\bmod m$, then the polynomial $f$ splits into linear factors over $\F_q$ if and only if it splits into linear factors over $\F_{q'}$.
\item The polynomial $f$ splits into linear factors over $\F_q$ if and only if $f$ has a root in $\F_q$.
\end{enumerate}
\end{prop}
\begin{proof} Let $m'=\lcm(q-1,m)$, and $L$ be the $m'$th cyclotomic field. Let $K=\Q(\alpha)$. Let $\mathfrak p$ be some prime ideal of $L$ above $p$, let $\mathcal O_p$ be the local ring and $k$ the residue field. Write 
\[ f(x)=a\prod_{i=1}^n (x-x_i),\]
where $x_i\in \mathcal O_p\cap K$ and $a$ lies in $\Z_{(p)}^\times$.

The polynomial $f$ splits over $\F_q$ if and only if the image of each of the $x_i$ in $k$ lies in $\F_q$, and this is true if and only if, for each $i$, the images of $x_i$ and $x_i^q$ in $k$ are the same. Letting $q=p^r$, let $\sigma$ denote the Frobenius at $p$, so that $f$ splits over $\F_q$ if and only if the images of $x_i$ and $\sigma^r(x_i)$ are the same in $k$.

Note that $\sigma^r(x_i)$ is some root of $f$, hence $\sigma^r(x_i)=x_j$ for some $j$. Since $f$ is separable modulo $p$ by assumption on $p$ (as the only ramified primes divide $m$), the images of the $x_i$ in $k$ are all distinct. Thus $f$ splits in $\F_q$ if and only if $\sigma^r$ fixes all of the $x_i$, so it lies in $\Gal(L/K)$.

But $\sigma^r$ lies in $\Gal(L/\Q)$, and is dependent only on its action on $\zeta_{m'}=\e^{2\pi\I/m'}$, and of course $\sigma^r(\zeta_{m'})=\zeta_{m'}^q$, which is in turn only dependent on the congruence of $q$ modulo $m'$. By the Chinese remainder theorem, this depends only on $q$ modulo $m$, i.e., whether $f$ splits in $\F_q$ depends only on the congruence of $q$ modulo $m$.

This completes the proof of (i).

\medskip

For the proof of (ii), note that, since $K/\Q$ is an abelian extension, the Galois group $\Gal(K/\Q)$ acts regularly on the roots of $f$. In particular, all non-trivial elements of $\Gal(K/\Q)$ act fixed-point freely on the roots of $f$. Since reduction modulo $p$ yields an embedding of the Galois group over $\F_p$ into the Galois group over $\Q$, this shows that the Galois group over $\F_p$ acts semi-regularly. In particular, if $f$ has a root over a particular finite field it splits over it.
\end{proof}

With this, we are able to prove Ennola duality for type $\tD_n$, swapping $\tD_n$ and ${}^2\tD_n$ for $n$ odd, and being a self-duality for $n$ even (as is the case for representations).

\begin{thm}\label{thm:EnnolatypeD} Given $H$ and $\chi$, assume that $\chi$ has indicator $+1$, and that $d$ is even. Let $\rho$ be a representation affording $\chi$, and suppose that $\rho$ is root cyclotomic.
\begin{enumerate}
\item There exist discriminating moduli for $H$ and $\chi$.
\item If $m$ is a discriminating modulus then $I_m^+$ is a subgroup of index at most $2$ in $I_m$.
\item If $4\mid d$ then $-1\in I_m^+$, and so $I_m^+$ is closed under taking negatives. If $d\equiv 2\bmod 4$ then $-1\in I_m^-$, and thus $I_m^-=\{-i:i \in I_m^+\}$.
\end{enumerate}
\end{thm}
\begin{proof} By \cite{pasechnik2021un} (proved independently by Guralnick--Navarro), $\rho$ can be chosen with image in $\Q(\zeta_n)\cap \R$ for some integer $n$. Let $\alpha$ be the determinant for the Gram matrix for a bilinear form associated to $\rho$. Since we are assuming that $\rho$ is root cyclotomic, $\sqrt\alpha$ lies in some cyclotomic field, which we can assume is $\Q(\zeta_n)$ by increasing $n$ if necessary. We see that $n$ is a defining modulus for $H$ and $\chi$. Since $\alpha$ is only defined up to a square, we may assume that $\alpha$ is integral.

Let $f$ be the minimal polynomial for $\sqrt\alpha$, and apply Proposition \ref{prop:allfinitefields}. This means that, for $q$ a power of $p\in\mP$, whether $f$ splits over $\F_q$ depends only on $q\bmod n$. Thus $\alpha$ being a square in $\F_q$ depends only on $q\bmod n$. If $n$ is not already a multiple of $4$, and $d\equiv 2\bmod 4$, multiply $n$ by $4$ so as to account for the change in discriminant of the form of $\GO_d^+(q)$ according to $q\bmod 4$. Since $n$ is already a defining modulus for $H$ and $\chi$, it must now be a discriminating modulus as well.

Let $J_1$ denote the subset of $I_n$ that consists of all elements $i$ such that $\zeta_n\mapsto \zeta_n^i$ fixes all values of $\chi$ \emph{and} fixes $\pm\sqrt\alpha$. Thus $J_1$ is a subgroup of index $2$ in $I_n$, and $J_1$ contains $-1$. To see this, notice that $\alpha$ is real (since the matrix is symmetric), and it is positive, as we noted at the start of this section. Thus $\pm\sqrt\alpha$ is a pair of real numbers, which are left invariant under complex conjugation, i.e., the field automorphism on $\Q(\zeta_n)$ such that $\zeta_n\mapsto\zeta_n^{-1}$.

If $4\mid d$ then $J_1=I_n^+$ and the theorem is proved. Thus we assume that $d\equiv 2\bmod 4$, and so we need to work further. Let $J_2$ consist of those elements of $I_n$ congruent to $1$ modulo $4$, another subgroup of index $2$, and not containing $-1$, so $J_1\neq J_2$. Since the determinant of the form for $\GO^+_d(q)$ is a square for $q\equiv 1\bmod 4$ and a non-square for $q\equiv 3\bmod 4$, $I_n^+$ consists of the union of $J_1\cap J_2$ and $I_n\setminus(J_1\cup J_2)$. Since $J_1$ and $J_2$ both have index $2$, $I_n^+$ is also a subgroup of index $2$ (it is the other overgroup of $J_1\cap J_2$) and does not contain $-1$.

We have therefore proved the second part of the theorem.
\end{proof}

Unfortunately, not all $H$ and $\chi$ are root cyclotomic, but counterexamples appear rare in low dimension.

\begin{example} Let $H$ be the Janko sporadic group $\mathrm J_1$ and $\chi$ be one of the two irreducible characters of degree $56$. Then $\Q(\chi)=\Q(\sqrt 5)=K$. Choosing a representation with image in $\GL_{56}(K)$, we find that the determinant $\alpha$ is difficult to write down except with a computer, but certainly the extension $\Q(\sqrt\alpha)/\Q$ is not Galois. To see this, note that if $\chi'$ is the other character of degree $56$, and letting $\alpha'$ correspond to $\chi'$, if $\alpha=a+b\sqrt 5$ then $\alpha'=a-b\sqrt 5$. If $\Q(\sqrt\alpha)/\Q$ were Galois then $\alpha$ would have a square root in $\F_p$ if and only if $\alpha'$ does, by Proposition \ref{prop:allfinitefields}. However, if $p=101$, then one of the $56$-dimensional representations of $H$ is conjugate to a subgroup of $\GO^+_{56}(101)$ and the other is conjugate to a subgroup of $\GO^-_{56}(101)$. The Galois group of $\Q(\sqrt\alpha,\sqrt{\alpha'})/\Q$ is dihedral of order $8$, so not abelian either.

For some primes there are two classes of subgroups $\mathrm J_1$ in $\GO^+_{56}(p)$, for some there are two in $\GO^-_{56}(p)$, and for some there is one class in each.
\end{example}

Thus there is in general no congruence for embedding simple groups for type $\tD_n$ with $n$ even, in contrast to the tables for small dimensions in \cite{bhrd,rogersphd,schroederphd}. Indeed, dimension $56$ appears to be the smallest dimension for which there is a representation that is not root cyclotomic, so there is no modulus by which we can distinguish the discriminant of the bilinear form.

However, for $\tD_n$ for $n$ odd, all representations for $\PSL_2(r)$ for $r\leq 31$, and other simple groups of dimension $d$ up to $250$ with $d\equiv 2\bmod 4$, are root cyclotomic. This leads to the following conjecture.

\begin{conj}\label{conj:EnnolatypeD} If $d\equiv 2\bmod 4$ and $\chi$ has indicator $+1$, if $\rho$ affords $\chi$ then $\rho$ is root cyclotomic.
\end{conj}

This conjecture holds if and only if Ennola duality holds for twice-odd dimensional orthogonal groups, which is expected.

For all type $\tD$ groups, even over $\Z$ so they are root cyclotomic, the modulus involved cannot obviously be read off from the character table, so it becomes another invariant of real representations that needs to be calculated. For example, for $HS$ and $McL$ in dimension $22$, the discriminants are $5$ and $15$ respectively.\footnote{One might be tempted to look at higher Frobenius--Schur indicators: for $HS$ and $McL$ they are the same for $r=3,5,11$, and are $0$, $1$ and $4$ respectively. The two groups differ for $r=7$, with values $3$ and $4$.}

\section{Ennola duality for exceptional groups}
\label{sec:exceptional}

Before we start this section, we underscore here that we will only be considering subgroups $H$ that occur in algebraic groups in characteristic $p$, where $p\nmid |H|$. In particular, the subgroup $H$ must arise in the corresponding algebraic group in characteristic $0$.

For exceptional groups, it is first not clear that there is such an analogue of the defining modulus. We will have to use the case-by-case lists of maximal subgroups of the finite exceptional groups to proceed. We also must consider what the analogue of an irreducible representation is, which we used when deciding which subgroups to embed into $\GL_d$. This is the reason we have to move away from the subset $\mP$, because there is no \emph{a priori} definition of it. (Of course, we can examine the tables of subgroups $H$ to produce the correct set $\mP$ for each subgroup $H$, but this \emph{ad hoc} treatment, for now, has little theoretical underpinning.)

We will only consider `Lie primitive' subgroups. Recall that a subgroup of a reductive algebraic group is \emph{irreducible} if it does not lie in a proper parabolic subgroup, and is \emph{Lie primitive} if it does not lie in a proper, positive-dimensional subgroup. Understanding Lie primitive subgroups, plus induction on the dimension of the reductive group, generally allows us to understand all irreducible subgroups. For our purposes of understanding maximal subgroups of finite groups, we restrict our attention to Lie primitive subgroups. (Note that we need not consider subgroups of the `same type', i.e., subgroups that are the fixed points of Steinberg endomorphisms on the same algebraic group -- e.g., subfield and twisted subfield subgroups -- because they do not arise from subgroups of the group over $\C$.)

If the ambient algebraic group $\mbG$ is of type $G_2$ or $F_4$ then everything works, defining moduli exist, and the set $I_n$ is closed under taking negatives, just as with types B and C. See Table \ref{tab:g2} for Lie primitive subgroups of $G_2(q)$, taken from \cite[Table 8.41]{bhrd}, and Table \ref{tab:f4} for Lie primitive subgroups of $F_4(q)$, taken from \cite{craven2020un}.

\begin{table}\begin{center}
\begin{tabular}{cc}
\hline Subgroup & $q$
\\\hline $2^3\cdot \PSL_3(2)$ & All $q$
\\ $\PSL_2(13)$ & $q\equiv \pm1,\pm3,\pm4\bmod 13$
\\ $\PSL_2(8)$ & $q\equiv \pm 1\bmod 9$
\\ $\PSU_3(3).2$ & All $q$
\\ \hline
\end{tabular}\caption{Lie primitive subgroups $H$ of $G_2(q)$, $\gcd(q,|H|)=1$.}\label{tab:g2}\end{center}
\end{table}

\begin{table}\begin{center}
\begin{tabular}{cc}
\hline Subgroup & $q$
\\\hline $3^3\cdot\SL_3(3)$ & All $q$
\\$\PSL_2(8)$ & $q\equiv \pm 1\bmod 7$
\\ $\PGL_2(13)$ & $q\equiv \pm 1\bmod 7$
\\ $\PSL_2(17)$ & $q\equiv \pm1,\pm2,\pm4,\pm8\bmod 17$
\\ $\PSL_2(25).2$ & All $q$
\\ $\PSL_2(27)$ & $q\equiv \pm1\bmod 7$
\\ \hline
\end{tabular}\caption{Lie primitive subgroups $H$ of $F_4(q)$, $\gcd(q,|H|)=1$.}\label{tab:f4}\end{center}
\end{table}

For $E_6(q)$ and ${}^2\!E_6(q)$, if one takes the simple group, or the triple cover, then there is no longer a defining modulus for the subgroup $\PSL_2(8).3$. This is a subgroup of the algebraic group, but whether it lies in the simple group depends not only on the congruence modulo $7$ ($\PSL_2(8)$ embeds in $E_6(q)$ if $q\equiv 1,2,4\bmod 7$), but on the presence of a cube root of $\sqrt{-7}+1$ in $\F_q$. Formally (see \cite[Theorem 6.8]{craven2020un}, the group $\PSL_2(8).3$ embeds in the finite simple group $E_6(q)$ if and only if $q\equiv 2\bmod3$, or $q$ is a cube, or $\sqrt{-7}+1$ has a cube root in $\F_q$. This latter condition cannot be expressed as a simple congruence modulo some number (to see this notice that the root does not lie in a cyclotomic field), so for simply connected groups, and almost simple subgroups, one has no defining modulus in general.

If, however, one uses adjoint versions of $E_6(q)$, so $E_6(q)$ with any diagonal automorphisms added on, then there are defining moduli, and the Lie primitive subgroups for $E_6(q)$ and ${}^2\!E_6(q)$ can be given in one table. The Lie primitive subgroups of ${}^\ep\!E_6(q)$ are in Table \ref{tab:epe6}: notice that we can switch between the two groups simply by replacing $q$ by $-q$.

\begin{table}\begin{center}
\begin{tabular}{cc}
\hline Subgroup & $q$
\\\hline $3^{3+3}\cdot \SL_3(3)$ & $3\mid(q-\ep1)$
\\ $\PSL_2(8).3$ & $\ep q\equiv 1,2,4\bmod 7$
\\ $\PSL_2(11)$ & $q\equiv \pm 1\bmod 5$, $\ep q\equiv 1,3,4,5,9\bmod 11$
\\ $\PSL_2(13)$ & $\ep q\equiv 3,5,6\bmod 7$, $q\equiv \pm2,\pm5,\pm6\bmod 13$
\\ $\PSL_2(19)$ & $q\equiv \pm1\bmod 5$, $\ep q\equiv 1,4,5,6,7,9,11,16,17\bmod 19$
\\ $\PGL_2(13)$ (nov) & $\ep q\equiv 1,2,4\bmod 7$, $q\equiv \pm1,\pm3,\pm4\bmod 13$
\\ \hline
\end{tabular}\caption{Lie primitive subgroups $H$ of the adjoint group of type ${}^\ep\!E_6(q)$ for $\ep=\pm$, $\gcd(q,|H|)=1$. The novelty Lie primitive subgroup $\PGL_2(13)$ occurs for an almost simple group inducing a graph automorphism on ${}^\ep\!E_6(q)$ (its derived subgroup is contained in $G_2$).}\label{tab:epe6}\end{center}
\end{table}

We also include a table of the known Lie primitive subgroups of $E_7(q)$, a list from \cite{craven2021un} that is expected to be complete. It displays the same behaviour, that replacing $q$ by $-q$ leaves the table invariant. We have again used the adjoint version: $\PSU_3(8).6$ and $\PGL_2(19)$ lie in the simple group (if they are in the adjoint group) if and only if $q\equiv \pm 1\bmod 8$. Note that all self-dual representations for $\PSU_3(3)$ are definable over $\Z$, so the defining modulus is $1$. However, for embedding in $E_7(q)$, we need $q\equiv\pm1 \bmod 8$. This examples shows that we do need to be careful when considering defining moduli for exceptional groups, even those that are Ennola self-dual.

\begin{table}\begin{center}
\begin{tabular}{cc}
\hline Subgroup & $q$
\\\hline $\PSU_3(3)$ & $q\equiv \pm 1\bmod 8$
\\$\PSU_3(8).6$ & All $q$
\\$\PSL_2(37)$ & $q\equiv \pm1,\pm3,\pm4,\pm7,\pm9,\pm10,\pm11,\pm12,\pm16\bmod 37$
\\$\PSL_2(29)$ & $q\equiv \pm1\bmod 5$, $q\equiv \pm1,\pm4,\pm5,\pm6,\pm7,\pm9,\pm13\bmod 29$
\\ $\PSL_2(27)$ & $q\equiv \pm 1\bmod 13$
\\ $\PSL_2(19)$ & $q\equiv\pm1\bmod 5$, $q_{19}\equiv q_3\bmod 2$
\\ $\PGL_2(19)$ & $q\equiv \pm1\bmod 5$
\\ \hline
\end{tabular}\caption{Known Lie primitive subgroups $H$ of the adjoint group of type $E_7(q)$, $\gcd(q,|H|)=1$. Here, $q_{19}$ is the order of $q$ modulo $19$ and $q_3$ is the order of $q$ modulo $3$.}\label{tab:e7}\end{center}
\end{table}


\begin{thebibliography}{10}

\bibitem{bhrd}
Bray, J., Holt, D., Roney-Dougal, C.: The maximal subgroups of low-dimensional finite classical groups. London Mathematical Society Lecture Note Series, no. 407, Cambridge University Press, Cambridge (2013)

\bibitem{craven2020un}
Craven, D.A.: The maximal subgroups of the exceptional groups {$F_4(q)$}, {$E_6(q)$} and {${}^2\!E_6(q)$} and related almost simple groups. Preprint (2020) arXiv:2103.04869

\bibitem{craven2021un}
Craven, D.A.: On the maximal subgroups of the exceptional groups {$E_7(q)$} and related almost simple groups. Preprint (2021)

\bibitem{ennola1963} Ennola, V.: On the characters of the finite unitary groups. Ann. Acad. Sci. Fenn. Ser. A No. 323, 35 pp. (1963)

\bibitem{abc}
Jansen, C., Lux, K., Parker, R., Wilson, R.: An atlas of {B}rauer characters. Oxford University Press, New York (1995)

\bibitem{kawanaka1985}
Kawanaka, N.: Generalized {G}elfand--{G}raev representations and {E}nnola duality. In: Algebraic groups and related topics (Kyoto/Nagoya, 1983), Adv. Stud. Pure Math., vol.~6, pp.~175--206, North-Holland, Amsterdam,  (1985)

\bibitem{malletesterman}
Malle, G., Testerman, D.: Linear algebraic groups and finite groups of {L}ie type. Cambridge University Press (2011)

\bibitem{pasechnik2021un}
Pasechnik, D.: Splitting fields of real irreducible representations of finite groups. Preprint (2021)  arXiv:2107.03452


\bibitem{rogersphd}
Rogers, D.: Maximal subgroups of classical groups in dimensions 16 and 17. Ph.D. thesis, University of Warwick (2017)

\bibitem{schroederphd}
Schr\"oder, A.: The maximal subgroups of classical groups in dimension 13, 14 and 15. Ph.D. thesis, University of St Andrews (2015)

\bibitem{turull1993}
Turull, A.: Schur index two and bilinear forms. J. Algebra 157, 562--572 (1993)

\end{thebibliography}
\end{document}